\documentclass[10pt, reqno]{amsart}
\usepackage{amsmath, amsthm, amscd, amsfonts, amssymb, graphicx, color}
\usepackage[bookmarksnumbered, colorlinks, plainpages]{hyperref}
\hypersetup{colorlinks=true,linkcolor=red, anchorcolor=green, citecolor=cyan, urlcolor=red, filecolor=magenta, pdftoolbar=true}
\usepackage{mathrsfs}

\newtheorem{theorem}{Theorem}[section]
\newtheorem{lemma}[theorem]{Lemma}

\newtheorem{corollary}[theorem]{Corollary}
\theoremstyle{definition}
\newtheorem{definition}[theorem]{Definition}
\newtheorem{example}[theorem]{Example}

\theoremstyle{remark}

\numberwithin{equation}{section}

\begin{document}
\setcounter{page}{1}

\title[quantum arithmetic of drinfeld modules]{Quantum arithmetic of Drinfeld modules}

\author[Nikolaev]
{Igor V. Nikolaev$^1$}

\address{$^{1}$ Department of Mathematics and Computer Science, St.~John's University, 8000 Utopia Parkway,  
New York,  NY 11439, United States.}
\email{\textcolor[rgb]{0.00,0.00,0.84}{igor.v.nikolaev@gmail.com}}

\dedicatory{All data are available as part of the manuscript}




\begin{abstract}
We study the quantum invariants of projective varieties 
 over  the number fields. Namely, explicit formulas for a functor $\mathscr{Q}$ on such varieties 
are proved.   The case of abelian varieties with complex multiplication 
is treated in detail. 
\end{abstract}

\maketitle

\section{Introduction}
Quantum arithmetic deals with a functor $\mathscr{Q}$ on the projective varieties $V(k)$ over
a number field $k$; we refer the reader to  Section 2.3 or \cite[Theorem 1.3]{Nik1} for the details.  
Such a functor  takes  values in the triples $(\Lambda, [I], K)$
consisting of a real number field $K$, an ideal class $[I]$ and an order $\Lambda\subset K$, i.e. a  subring 
 of  the ring of integers  of  $K$  [Handelman 1981] \cite{Han1}.  
 The invariant $(\Lambda, [I], K)$ comes from the $K$-theory 
 of operator algebras  related to the 
 quantum mechanics  [Blackadar 1986] \cite{B}; hence the name.   
 The existence of functor   $\mathscr{Q}$ was  proved by contradiction \cite[Remark 1.4]{Nik1}.  
  Such a  proof entails  no  efficient formula for the number field $K$ in terms of the field of definition of variety $V(k)$,  
  except for the special case of complex multiplication \cite[Theorem 4.1]{Nik1}. 
The aim of our note is to fill the gap by establishing such a formula  (Theorem \ref{thm1.1}).
We shall use the following notation and facts. 

Let $\mathfrak{k}:=\mathbf{F}_{p^n}$ be a finite field and $\tau(x)=x^p$. Consider a ring $\mathfrak{k}\langle\tau\rangle$
of the non-commutative polynomials given by the commutation relation $\tau a=a^p\tau$ for all $a\in A$, where $A:=\mathfrak{k}[T]$
is the ring of polynomials in  variable $T$ over $\mathfrak{k}$.  
The Drinfeld module $Drin_A^r(\mathfrak{k})$ of rank $r\ge 1$ is a homomorphism 
\begin{equation}
\rho: A\buildrel r\over\longrightarrow \mathfrak{k}\langle\tau\rangle
\end{equation}
given by a polynomial $\rho_a=a+c_1\tau+\dots+c_r\tau^r$,
where $a\in A$ and $c_i\in \mathfrak{k}$ 
[Rosen 2002] \cite[Section 12]{R}. 
An isogeny between Drinfeld modules  $Drin_A^{r}(\mathfrak{k})$ and $\widetilde{Drin}_A^{r}(\mathfrak{k})$
is a surjective morphism 
$f: Drin_A^{r}(\mathfrak{k})\to \widetilde{Drin}_A^{r}(\mathfrak{k})$ with a finite kernel,  \textit{ibid}. 
Consider a torsion submodule $\Lambda_{\rho}[a]:=\{\lambda\in\overline{\mathfrak{k}} ~|~\rho_a(\lambda)=0\}$ of
the $A$-module $\overline{\mathfrak{k}}$. 
 Drinfeld modules $Drin_A^r(\mathfrak{k})$ and associated torsion submodules $\Lambda_{\rho}[a]$ 
 define generators of a non-abelian class field theory for the function fields. 
 Namely, for each non-zero $a\in A$ the function 
field $\mathfrak{k}(T)\left(\Lambda_{\rho}[a]\right)$  is a Galois extension of the field  $\mathfrak{k}(T)$
of rational functions in variable $T$ over $\mathfrak{k}$,
such that its  Galois group is isomorphic to a subgroup of the matrix group $GL_r\left(A/aA\right)$
[Rosen 2002] \cite[Proposition 12.5]{R}. 

On the other hand, the norm closure of a representation of the multiplicative semi-group of the ring 
$\mathfrak{k}\langle\tau\rangle$ [Li 2017] \cite{Li1} 
by  bounded linear operators on a Hilbert space gives rise to the noncommutative torus $\mathscr{A}_{RM}^{2r}$ having
real multiplication (RM)  \cite{Nik2}. 
The latter is   a $C^*$-algebra generated by the unitary operators $u_1,\dots, u_{2r}$ 
satisfying the commutation relations $\{u_ju_i=e^{2\pi i\theta_{ij}}u_iu_j ~|~1\le i,j\le 2r\}$,
where $\theta_{ij}$ are algebraic numbers and  $\Theta=(\theta_{ij})\in M_{2r}(\mathbf{R})$ is a 
skew-symmetric matrix  [Rieffel 1990] \cite{Rie1}.
The $K$-theory of the $C^*$-algebra  $\mathscr{A}_{RM}^{2r}$ is well known 
[Blackadar 1986] \cite[Chapter III]{B} and [Rieffel 1990] \cite[Section 3]{Rie1}.
Namely, the Grothendieck semi-group is given by the formula  $K_0^+(\mathscr{A}_{RM}^{2r})\cong \mathbf{Z}+\alpha_1\mathbf{Z}+\dots+
\alpha_{r}\mathbf{Z}\subset \mathbf{R}$, where $\alpha_i$ are algebraic integers of degree $2r$ over $\mathbf{Q}$
\cite[Section 2.2.2]{Nik2}. 
 The following is true  \cite[Theorem 3.3]{Nik2}:  (i)  there exists a functor $F: Drin_A^{r}(\mathfrak{k})\mapsto \mathscr{A}_{RM}^{2r}$
from the category of Drinfeld  modules $\mathfrak{D}$ to a category 
of the noncommutative tori $\mathfrak{A}$,   which maps any pair of isogenous  (isomorphic, resp.) 
modules  $Drin_A^{r}(\mathfrak{k}), ~\widetilde{Drin}_A^{r}(\mathfrak{k})\in \mathfrak{D}$
to a pair of the homomorphic (isomorphic, resp.)  tori  $\mathscr{A}_{RM}^{2r}, \widetilde{\mathscr{A}}_{RM}^{2r}
\in \mathfrak{A}$,  (ii) 
$F(\Lambda_{\rho}[a])=\{e^{2\pi i\alpha_i+\log\log\varepsilon} ~|~1\le i\le r\}$,
where $\mathscr{A}_{RM}^{2r}=F(Drin_A^r(\mathfrak{k}))$ and   $\log\varepsilon$ is a scaling factor 
and (iii) the number field $k=\mathbf{Q}(F(\Lambda_{\rho}[a]))$ is the extension of its subfield 
with the Galois group $G\subseteq GL_r\left(A/aA\right)$.

 An isogeny  $f:Drin_A^{r}(\mathfrak{k})\to\widetilde{Drin}_A^{r}(\mathfrak{k})$ defines a Grothendieck semi-group
 \linebreak
 $K_0^+(\widetilde{\mathscr{A}}_{RM}^{2r})=\{\mathbf{Z}+\frac{\alpha_1}{m_1}\mathbf{Z}+\dots+\frac{\alpha_{r}}{m_r}\mathbf{Z} ~|~m_i\in \mathbf{N}\}$
 (Lemma \ref{lm3.1} and Corollary \ref{cor3.2}) 
 and an extension of the number field $k\cong\mathbf{Q}(F(\Lambda_{\rho}[a]))$ (Lemma \ref{lm3.3}). 
 We use these facts to construct an (\'etale) branched covering $V(k)\to \mathbf{C}P^n$ [Namba 1985] \cite[Theorem 5]{Nam1}
 of the $n$-dimensional projective space $\mathbf{C}P^n$ (Corollary \ref{cor3.4}). 
 Denote by $\log k$ ($\arccos k$, resp.) a number field 
generated by $\alpha_i$, such that $k\cong \mathbf{Q}(e^{2\pi i\alpha_i+\log\log\varepsilon})$  
($k\cong \mathbf{Q}(\cos 2\pi\alpha_i \times\log\varepsilon)$, resp.)  
(Corollary \ref{cor2.2}). Our main result can be formulated as follows. 
\begin{theorem}\label{thm1.1}
\begin{equation}\label{eq1.1}
\mathscr{Q}(V(k))=
\begin{cases} (\Lambda, [I], ~\log k), & if ~k\subset\mathbf{C} - \mathbf{R},\cr
                (\Lambda, [I], ~\arccos k), & if ~k\subset\mathbf{R}.
\end{cases}               
\end{equation}
\end{theorem}

\bigskip
We start with  a brief review of the preliminary facts
 in Section 2. Theorem \ref{thm1.1} 
is proved in Section 3.  
The case of abelian varieties with complex multiplication is treated 
in Section 4.

\section{Preliminaries}
We briefly review noncommutative tori, Drinfeld modules and quantum arithmetic. 
We refer the reader to [Blackadar 1986] \cite{B},  [Rieffel 1990] \cite{Rie1},  [Rosen 2002] \cite[Chapters 12 \& 13]{R} 
and \cite{Nik1}  for a detailed exposition.

\subsection{Noncommutative geometry}
\subsubsection{$C^*$-algebras}
The $C^*$-algebra is an algebra  $\mathscr{A}$ over $\mathbf{C}$ with a norm 
$a\mapsto ||a||$ and an involution $\{a\mapsto a^* ~|~ a\in \mathscr{A}\}$  such that $\mathscr{A}$ is
complete with  respect to the norm, and such that $||ab||\le ||a||~||b||$ and $||a^*a||=||a||^2$ for every  $a,b\in \mathscr{A}$.  
Each commutative $C^*$-algebra is  isomorphic
to the algebra $C_0(X)$ of continuous complex-valued
functions on some locally compact Hausdorff space $X$. 
Any other  algebra $\mathscr{A}$ can be thought of as  a noncommutative  
topological space. 

\subsubsection{K-theory of $C^*$-algebras}
By $M_{\infty}(\mathscr{A})$ 
one understands the algebraic direct limit of the $C^*$-algebras 
$M_n(\mathscr{A})$ under the embeddings $a\mapsto ~\mathbf{diag} (a,0)$. 
The direct limit $M_{\infty}(\mathscr{A})$  can be thought of as the $C^*$-algebra 
of infinite-dimensional matrices whose entries are all zero except for a finite number of the
non-zero entries taken from the $C^*$-algebra $\mathscr{A}$.
Two projections $p,q\in M_{\infty}(\mathscr{A})$ are equivalent, if there exists 
an element $v\in M_{\infty}(\mathscr{A})$,  such that $p=v^*v$ and $q=vv^*$. 
The equivalence class of projection $p$ is denoted by $[p]$.   
We write $V(\mathscr{A})$ to denote all equivalence classes of 
projections in the $C^*$-algebra $M_{\infty}(\mathscr{A})$, i.e.
$V(\mathscr{A}):=\{[p] ~:~ p=p^*=p^2\in M_{\infty}(\mathscr{A})\}$. 
The set $V(\mathscr{A})$ has the natural structure of an abelian 
semi-group with the addition operation defined by the formula 
$[p]+[q]:=\mathbf{diag}(p,q)=[p'\oplus q']$, where $p'\sim p, ~q'\sim q$ 
and $p'\perp q'$.  The identity of the semi-group $V(\mathscr{A})$ 
is given by $[0]$, where $0$ is the zero projection. 
By the $K_0$-group $K_0(\mathscr{A})$ of the unital $C^*$-algebra $\mathscr{A}$
one understands the Grothendieck group of the abelian semi-group
$V(\mathscr{A})$, i.e. a completion of $V(\mathscr{A})$ by the formal elements
$[p]-[q]$.  The image of $V(\mathscr{A})$ in  $K_0(\mathscr{A})$ 
is a positive cone $K_0^+(\mathscr{A})$ defining  the order structure $\le$  on the  
abelian group  $K_0(\mathscr{A})$. The pair   $\left(K_0(\mathscr{A}),  K_0^+(\mathscr{A})\right)$
is known as a dimension group of the $C^*$-algebra $\mathscr{A}$.

\subsubsection{Noncommutative tori}
The $m$-dimensional noncommutative torus $\mathscr{A}_{\Theta}^m$ is the
universal $C^*$-algebra  generated by unitary operators $u_1,\dots, u_m$
satisfying the commutation relations 
\begin{equation}\label{eq2.1}
u_ju_i=e^{2\pi i \theta_{ij}} u_iu_j, \quad 1\le i,j\le m
\end{equation}
for a skew-symmetric matrix  $\Theta=(\theta_{ij})\in M_m(\mathbf{R})$
[Rieffel 1990] \cite{Rie1}. 
 It is known that 
  $K_0(\mathscr{A}_{\Theta}^m)\cong K_1(\mathscr{A}_{\Theta}^m)\cong \mathbf{Z}^{2^{m-1}}$.
The canonical trace $\tau$ on the $C^*$-algebra
$\mathscr{A}_{\Theta}^m$ defines a homomorphism from 
$K_0(\mathscr{A}_{\Theta}^m)$ to the real line $\mathbf{R}$;
under the homomorphism, the image of $K_0(\mathscr{A}_{\Theta}^m)$
is a $\mathbf{Z}$-module, whose generators $\tau=(\tau_i)$ are polynomials 
in $\theta_{ij}$.  The noncommutative
tori  $\mathscr{A}_{\Theta}^m$ and $\mathscr{A}_{\Theta'}^m$ are Morita
equivalent,  if  the matrices $\Theta$ and $\Theta'$
belong to the same orbit of a subgroup $SO(m,m~|~\mathbf{Z})$ of the
group $GL_{2m}(\mathbf{Z})$, which acts on $\Theta$ by the formula
$\Theta'=(A\Theta+B)~/~(C\Theta+D)$, where $(A, B,  C,  D)\in GL_{2m}(\mathbf{Z})$
and  the matrices $A,B,C,D\in GL_m(\mathbf{Z})$ satisfy the conditions
$A^tD+C^tB=I,\quad A^tC+C^tA=0=B^tD+D^tB$,
where $I$ is the unit matrix and $t$ at the upper right of a matrix 
means a transpose of the matrix.)  
The group $SO(m, m ~| ~\mathbf{Z})$ can be equivalently defined as a
subgroup of the group  $SO(m, m ~| ~\mathbf{R})$ consisting of linear transformations 
of the space $\mathbf{R}^{2m}$,  which 
preserve the quadratic form $x_1x_{m+1}+x_2x_{k+2}+\dots+x_kx_{2m}$.

\subsection{Non-abelian class field theory}
Let  $\mathfrak{k}:=\mathbf{F}_q(T)$ ($A:=\mathbf{F}_q[T]$, resp.) be the field of rational functions (the ring of polynomial functions, resp.)
in one variable $T$ over a finite field $\mathbf{F}_q$, where $q=p^n$
and let  $\tau_p(x)=x^p$. 
Recall that  the  Drinfeld module $Drin_A^{r}(\mathfrak{k})$   of rank $r\ge 1$
is a homomorphism
\begin{equation}\label{eq2.2}
\rho:  ~A\buildrel r\over\longrightarrow \mathfrak{k}\langle\tau_p\rangle
\end{equation}
given by a polynomial $\rho_a=a+c_1\tau_p+c_2\tau_p^2+\dots+c_r\tau_p^r$ with $c_i\in \mathfrak{k}$ and $c_r\ne 0$, 
such that for all $a\in A$ the constant term of $\rho_a$ is $a$ and 
$\rho_a\not\in \mathfrak{k}$ for at least one $a\in A$ [Rosen 2002] \cite[p. 200]{R}.
For each non-zero $a\in A$ the function 
field $\mathfrak{k}\left(\Lambda_{\rho}[a]\right)$  is a Galois extension of $\mathfrak{k}$,
such that its  Galois group is isomorphic to a subgroup $G$ of the matrix group $GL_r\left(A/aA\right)$,
where   $\Lambda_{\rho}[a]=\{\lambda\in\overline{ \mathfrak{k}} ~|~\rho_a(\lambda)=0\}$
is a torsion submodule of the non-trivial  Drinfeld module  $Drin_A^{r}(\mathfrak{k})$  [Rosen 2002] \cite[Proposition 12.5]{R}.
Clearly, the abelian extensions correspond to the case $r=1$.

Let $G$ be a  left cancellative  semigroup generated by $\tau_p$ and all  $a_i\in \mathfrak{k}$ subject to the commutation relations 
$\tau_p a_i=a_i^p\tau_p$.
\footnote{In other words, we omit the additive structure and consider a multiplicative semigroup of the ring $\mathfrak{k}\langle\tau_p\rangle$.  }
  Let $C^*(G)$ be the semigroup $C^*$-algebra [Li 2017] \cite{Li1}.  
For a Drinfeld module  $Drin_A^{r}(\mathfrak{k})$  defined  by  (\ref{eq2.2}) we consider a homomorphism of the semigroup $C^*$-algebras:  
\begin{equation}\label{eq2.3}
C^*(A)\buildrel r\over\longrightarrow C^*(\mathfrak{k}\langle\tau_p\rangle). 
\end{equation}
It is proved that (\ref{eq2.3}) defines a map  $F: Drin_A^{r}(\mathfrak{k})\mapsto \mathscr{A}_{RM}^{2r}$ \cite[Definition 3.1]{Nik2}. 
\begin{theorem}\label{thm2.1}
{\bf (\cite{Nik2})}
The following is true:

\medskip
(i) the map $F: Drin_A^{r}(\mathfrak{k})\mapsto \mathscr{A}_{RM}^{2r}$ is a functor 
from the category of Drinfeld  modules $\mathfrak{D}$ to a category 
of the noncommutative tori $\mathfrak{A}$,   which maps any pair of isogenous  (isomorphic, resp.) 
modules  $Drin_A^{r}(\mathfrak{k}), ~\widetilde{Drin}_A^{r}(\mathfrak{k})\in \mathfrak{D}$
to a pair of the homomorphic (isomorphic, resp.)  tori  $\mathscr{A}_{RM}^{2r}, \widetilde{\mathscr{A}}_{RM}^{2r}
\in \mathfrak{A}$;  

\smallskip
(ii) $F(\Lambda_{\rho}[a])=\{e^{2\pi i\alpha_i+\log\log\varepsilon} ~|~1\le i\le r\}$,
where $\mathscr{A}_{RM}^{2r}=F(Drin_A^{r}(\mathfrak{k}))$, 
$\alpha_i$ are generators of the Grothendieck semi-group $K_0^+(\mathscr{A}_{RM}^{2r})$,  $\log\varepsilon$ is a scaling factor
 and $\Lambda_{\rho}(a)$ is the  torsion submodule of the $A$-module $\overline{\mathfrak{k}_{\rho}}$;

\smallskip
(iii) the Galois group $Gal \left(k_0(e^{2\pi i\alpha_i+\log\log\varepsilon})  ~| ~k_0\right)\subseteq GL_{r}\left(A/aA\right)$,
where $k_0$ is a subfield of the number field $\mathbf{Q}(e^{2\pi i\alpha_i+\log\log\varepsilon})$. 
 \end{theorem}
Theorem \ref{thm2.1} implies a non-abelian class field theory as follows.
Fix a non-zero $a\in A$ and let $G:=Gal~(\mathfrak{k}(\Lambda_{\rho}[a]) ~|~ \mathfrak{k})\subseteq GL_r(A/aA)$,
where  $\Lambda_{\rho}[a]$ is the torsion submodule of the $A$-module  $\overline{\mathfrak{k}_{\rho}}$.
Consider the number field $k=\mathbf{Q}(F(\Lambda_{\rho}[a]))$. 
Denote by $k_0$ the maximal subfield of $k$ which is fixed by the action of 
all elements of the group $G$. 
\begin{corollary}\label{cor2.2} 
{\bf (Non-abelian class field theory)} 
The number field
\begin{equation}\label{eq2.4}
k\cong
\begin{cases} k_0\left(e^{2\pi i\alpha_i +\log\log\varepsilon}\right), & if ~k_0\subset\mathbf{C} - \mathbf{R},\cr
               k_0\left(\cos 2\pi\alpha_i \times\log\varepsilon\right), & if ~k_0\subset\mathbf{R},
\end{cases}               
\end{equation}
is a Galois extension of  $k_0$,
 such that  $Gal~(k |k_0 )\cong G$.
\end{corollary}

\subsection{Quantum arithmetic}
Let $V$ be an $n$-dimensional projective variety over the field of complex numbers $\mathbf{C}$.
 Recall \cite[Section 5.3.1]{N} that the Serre $C^*$-algebra  $\mathscr{A}_V$
 is  the norm closure of a self-adjoint representation of the twisted 
 homogeneous coordinate ring of  $V$  by the bounded linear operators acting on a Hilbert space;
 we refer the reader to  [Stafford \& van ~den ~Bergh 2001] \cite{StaVdb1} or Section 2.1 for the details.  
 Let $(K_0(\mathscr{A}_V), K_0^+(\mathscr{A}_V))$ be a dimension group 
 of the $C^*$-algebra $\mathscr{A}_V$ [Blackadar 1986] \cite[Section 6.1]{B}. 
  The triple  $(\Lambda, [I], K)$ stays for a dimension group generated by
   the ideal class $[I]$ of an order $\Lambda\subseteq O_K$ 
  in the ring of integers of a number field $K$ [Effros 1981] \cite[Chapter 6]{E},
  [Handelman 1981] \cite{Han1} or \cite[Theorem 3.5.4]{N}.  
  \begin{definition}\label{dfn1.1}
  The Serre $C^*$-algebra $\mathscr{A}_V$ is said to have real multiplication  by 
  the triple
   $(\Lambda, [I], K)$,   if  there exists  an isomorphism of the dimension group
   $(K_0(\mathscr{A}_V), K_0^+(\mathscr{A}_V))\cong (\Lambda, [I], K)$,
   where   $\Lambda\subseteq O_K$ is an order,   $[I]\subset\Lambda$ is  an ideal class  and 
   $K$ is  a  number field. 
  \end{definition}
\begin{example}\label{exm2.4}
{\bf (\cite{Nik1})}
Let $V\cong\mathscr{E}_{CM}$ be an elliptic curve with complex multiplication by the triple $(L, [I], k)$,
where $L=\mathbf{Z}+fO_k$ is an order of conductor $f\ge 1$ in the ring $O_k$ of an
imaginary quadraitic field $k\cong \mathbf{Q}(\sqrt{-d})$ and $[I]$ an ideal class in $L$.   
Then  $\mathscr{A}_V\cong \mathscr{A}_{RM}$ is a noncommutative torus with real 
multiplication by the triple $(\Lambda, [I], K)$,
where $[I]$ an ideal class in the order $L=\mathbf{Z}+f'O_k$ of conductor $f'\ge 1$ in the ring  $O_k$ of an
imaginary quadraitic field $K\cong\mathbf{Q}(\sqrt{d})$, such that 
$f'$ is the least integer satisfying a group isomorphism
$Cl(\mathbf{Z}+f'O_{K})\cong Cl(\mathbf{Z}+fO_{k})$,
where  $Cl (R)$ is the class group of the ring $R$. 
\end{example}

Let $\mathscr{M}_V$ be the deformation moduli space of variety $V$ and $m=\dim_{\mathbf{C}} \mathscr{M}_V$. 
 Denote by $O_K$ the ring of integers of a number field  $K$ of degree 
  $\deg ~(K|\mathbf{Q})=2m$. 
\begin{theorem}\label{thm2.5}
{\bf (\cite{Nik1})}
 The Serre $C^*$-algebra $\mathscr{A}_V$ has real multiplication 
by a triple $(\Lambda, [I], K)$, if and only if, 
the projective variety $V$ is defined over a number field $k$.
\end{theorem}

\section{Proof of Theorem \ref{thm1.1}}
The idea of proof was outlined in Section 1.  For the sake of clarity, let us
review the main steps. 
Each isogeny of the Drinfeld module $Drin_A^{r}(\mathfrak{k})$ generates
an endomorphism of the Grothendieck semi-group 
$K_0^+(\mathscr{A}_{RM}^{2r})\cong \mathbf{Z}+\alpha_1\mathbf{Z}+\dots+
\alpha_{r}\mathbf{Z}\subset \mathbf{R}$,  where $\mathscr{A}_{RM}^{2r}=F(Drin_A^{r}(\mathfrak{k}))$
(Lemma \ref{lm3.1}). 
These endomorphisms are classified by the $r$-tuples of integers $m_i\ge 1$ (Corollary \ref{cor3.2}). 
In particular,  each isogeny defines an extension of the field   $k=\mathbf{Q}(e^{2\pi i\alpha_i+\log\log\varepsilon})$
by the $r$ roots of degree $m_i$ (Lemma \ref{lm3.3}). 
Using results of [Namba 1985] \cite[Theorem 5]{Nam1} we construct a branched covering $V(k)\to \mathbf{C}P^n$
 of the $n$-dimensional projective space $\mathbf{C}P^n$ (Corollary \ref{cor3.4}). 
The $V(k)$ is proved to satisfy equation (\ref{eq1.1}) (Lemma \ref{lm3.5}). 
We pass to a detailed argument. 
\begin{lemma}\label{lm3.1}
 Drinfeld modules $Drin_A^{r}(\mathfrak{k})$ and  $\widetilde{Drin}_A^{r}(\mathfrak{k})$
 are isogenous, if and only if, $\widetilde{\Lambda}\subseteq \Lambda$, $[I]\subseteq\widetilde{[I]}$
 and $K\cong\widetilde{K}$, where $\Lambda$ ($\widetilde{\Lambda}$, resp.) is the endomorphism ring of $K_0^+(F(Drin_A^{r}(\mathfrak{k})))\cong \mathbf{Z}+\sum_{k=1}^{r}\alpha_i\mathbf{Z}$
  ($K_0^+(F(\widetilde{Drin_A}^{r}(\mathfrak{k})))$, resp.),  $[I]\subseteq\widetilde{[I]}$ are the ideal classes of orders  $\Lambda$ and $\widetilde{\Lambda}$
   belonging to  the same coset in the  inclusion of the corresponding ideal class groups 
  and  $K=\mathbf{Q}(\alpha_i)$.
    \end{lemma} 
\begin{proof}
(i) In view of item (i) of Theorem \ref{thm2.1}, any pair of isogenous Drinfeld modules 
 $Drin_A^{r}(\mathfrak{k})$ and  $\widetilde{Drin}_A^{r}(\mathfrak{k})$
defines  a homomorphism of the $C^*$-algebras  $h: \mathscr{A}_{RM}^{2r} \to \widetilde{\mathscr{A}}_{RM}^{2r}$,
where  $\mathscr{A}_{RM}^{2r}=F(Drin_A^{r}(\mathfrak{k}))$ and  $\widetilde{\mathscr{A}}_{RM}^{2r}=F(\widetilde{Drin}_A^{r}(\mathfrak{k}))$. 

\medskip
(ii) Consider an induced order-homomorphism $h_*: K_0^+(\mathscr{A}_{RM}^{2r})\to K_0^+(\widetilde{\mathscr{A}}_{RM}^{2r})$
of the Grothendieck semi-groups of  $\mathscr{A}_{RM}^{2r}$ and  $\widetilde{\mathscr{A}}_{RM}^{2r}$ \cite[Section 6]{B}.

\medskip
(iii) Recall that $K_0^+(\mathscr{A}_{RM}^{2r})\cong \mathbf{Z}+\alpha_1\mathbf{Z}+\dots+
\alpha_{r}\mathbf{Z}\subset \mathbf{R}$ and $K_0^+(\widetilde{\mathscr{A}}_{RM}^{2r})\cong \mathbf{Z}+\widetilde{\alpha}_1\mathbf{Z}+\dots+
\widetilde{\alpha}_{r}\mathbf{Z}\subset \mathbf{R}$, where $\alpha_i$ and  $\widetilde{\alpha}_k$ are algebraic numbers of degree $2r$
over $\mathbf{Q}$. 
In view of (ii),  one gets $K_0^+(\widetilde{\mathscr{A}}_{RM}^{2r})\subseteq K_0^+(\mathscr{A}_{RM}^{2r})$ and, therefore,
$K\cong \widetilde{K}$, where $K=\mathbf{Q}(\alpha_i)$ and $\widetilde{K}=\mathbf{Q}(\widetilde{\alpha}_k)$. 

\medskip
(iv) On the other hand, we have $\Lambda:=End~ (K_0^+(\mathscr{A}_{RM}^{2r}))$ and 
$\widetilde{\Lambda}:=End~ (K_0^+(\widetilde{\mathscr{A}}_{RM}^{2r}))$, where $End$ is the ring of endomorphisms
of the respective Grothendieck semi-groups.  Since  $K_0^+(\widetilde{\mathscr{A}}_{RM}^{2r})\subseteq K_0^+(\mathscr{A}_{RM}^{2r})$,
one gets an inclusion of the rings $\widetilde{\Lambda}\subseteq\Lambda$. 

\medskip
(v) Recall that $\Lambda$ and $\widetilde{\Lambda}$ are orders in the ring of integers $O_K$ of the number field $K=\mathbf{Q}(\alpha_i)$. 
Denote by $Cl ~(\Lambda)$ and $Cl~(\widetilde{\Lambda})$ the respective ideal class groups. Since  $\widetilde{\Lambda}\subseteq\Lambda$,
one gets a reverse inclusion of the finite abelian groups  $Cl ~(\Lambda)\subseteq Cl~(\widetilde{\Lambda})$. 
If $[I]\in Cl~(\Lambda)$, then $[I]\in Cl~(\widetilde{\Lambda})$ and the corresponding element  $\widetilde{[I]}\in Cl~(\widetilde{\Lambda})$
must be in the same coset with $[I]$ defined by the subgroup $Cl ~(\Lambda)$ of the ideal class group $Cl~(\widetilde{\Lambda})$.

\medskip
(vi) The necessary conditions of Lemma \ref{lm3.1} are established likewise and the proof is left to the reader. 

\bigskip
Lemma \ref{lm3.1} is proved. 
\end{proof}

\begin{corollary}\label{cor3.2}
The set of all isogenies of the Drinfeld modules $Drin_A^{r}(\mathfrak{k})\to\widetilde{Drin}_A^{r}(\mathfrak{k})$
is bijective with the integer tuples $\{(m_1,\dots, m_{r}) ~|~m_i\ge 1\}$. Moreover, the composition of isogenies correspond to a 
component-wise multiplication of the tuples.
 \end{corollary} 
\begin{proof}
Each  $\widetilde{\Lambda}\subseteq \Lambda$ has the form 
$\widetilde{\Lambda}=\mathbf{Z}+\widetilde{\alpha}_1\mathbf{Z}+\dots+\widetilde{\alpha}_{r}\mathbf{Z}$,
where $\widetilde{\alpha}_i=\frac{\alpha_i}{m_i}$ 
for some integers $m_i\ge 1$\cite[p. 88]{BS}. 
It is easy to see, that a composition of isogenies   $Drin_A^{r}(\mathfrak{k})\to\widetilde{Drin}_A^{r}(\mathfrak{k})
\to \widetilde{\widetilde{Drin}}_A^{r}(\mathfrak{k})$ corresponds to the inclusions $\widetilde{\widetilde{\Lambda}}\subseteq\widetilde{\Lambda}\subseteq \Lambda$.
The latter gives rise to a component-wise multiplication of the integer vectors $(m_1,\dots,m_{r})$. Corollary \ref{cor3.2} follows.
\end{proof}

\begin{lemma}\label{lm3.3}
An isogeny $Drin_A^{r}(\mathfrak{k})\to\widetilde{Drin}_A^{r}(\mathfrak{k})$
acts on the image  $F(\Lambda_{\rho}[a])=\{e^{2\pi i\alpha_i+\log\log\varepsilon} ~|~1\le i\le r\}$
of the torsion submodule $\Lambda_{\rho}[a]$ via an extension
$\widetilde{k}=k(\sqrt[m_1]{x},\dots,\sqrt[m_{r}]{x})$
($\widetilde{k}=k(\Re\sqrt[m_1]{x},\dots,\Re\sqrt[m_{r}]{x})$ , resp.)
of the imaginary (real, resp.) number field $k=\mathbf{Q}(e^{2\pi i\alpha_i+\log\log\varepsilon})$
($k=\mathbf{Q}(\cos 2\pi\alpha_i \times\log\varepsilon)$, resp.) 
by the roots of degree $m_i$ of the elements $x\in k$. 
\end{lemma} 
\begin{proof}
(i) Using Corollary \ref{cor3.2}, one obtains a formula for the image of torsion submodule $\widetilde{\Lambda}_{\rho}[a]$:
\begin{equation}\label{eq3.1}
F(\widetilde{\Lambda}_{\rho}[a])=\{e^{2\pi \frac{\alpha_i}{m_i}+\log\log\widetilde{\varepsilon}} ~|~1\le i\le r\},
\end{equation}
where $m_i\ge 1$ are integer numbers. We let $m:=LCM (m_1,\dots, m_{r})$ be the least common multiple and
 $\log\widetilde{\varepsilon}:=(\log\varepsilon)^m$. 

\medskip
(ii) The substitution $\log \widetilde{\varepsilon}_i=(\log\varepsilon)^{\frac{1}{m_i}}$ brings (\ref{eq3.1}) 
to the form:
\begin{equation}\label{eq3.2}
F(\widetilde{\Lambda}_{\rho}[a])=\{\left(e^{2\pi  i\alpha_i+\log\log\varepsilon}\right)^{\frac{1}{m_i}}  ~|~1\le i\le r\}. 
\end{equation}
It follows from (\ref{eq3.2}) that the field $k$ is an extension of degree $m$ of the field $\widetilde{k}$ 
by the roots $\{\sqrt[m_i]{x} ~|~1\le i\le r\}$,  where $x\in k$. 

\medskip
(iii)  If $k=\mathbf{Q}(\cos 2\pi\alpha_i \times\log\varepsilon)$ is a real number field, we
take the real part $\Re$ of the complex numbers in formulas (\ref{eq3.1}) 
and (\ref{eq3.2}) and repeat the  argument of steps (i) and (ii).
Lemma \ref{lm3.3} is proved.
\end{proof}

\begin{corollary}\label{cor3.4}
For every $n\ge 1$ and every isogeny  $Drin_A^{r}(\mathfrak{k})\to\widetilde{Drin}_A^{r}(\mathfrak{k})$
there exists a Galois covering $V(k)\to kP^n$ with the branch locus of index $(m_1,\dots, m_{r})$. 
\end{corollary} 
\begin{proof}
The result follows from [Namba 1985]  \cite{Nam1}.   
Indeed, given an integer $n\ge 1$ and the index of multiplicities $(m_1,\dots, m_{r})$ of the branch locus 
consisting of the union of $r$ hyper-surfaces,  one can construct a Galois covering: 
\begin{equation}
V\longrightarrow \mathbf{C}P^n,
\end{equation}
where $\mathbf{C}P^n$ is the $n$-dimensional complex projective space 
[Namba 1985]  \cite[Theorem 5]{Nam1}.
Since $k\subset\mathbf{C}$, one can always assume that projective variety
$V$ is defined over the number field $k$,  for otherwise one takes  a deformation of $V$ 
in its moduli space. Therefore  we get a covering:   
\begin{equation}
V(k)\longrightarrow \mathbf{C}P^n.
\end{equation}
\end{proof}

\begin{lemma}\label{lm3.5}
The projective variety $V(k)$ satisfies equation (\ref{eq1.1}). 
\end{lemma} 
\begin{proof}
(i) Recall that  $\log k:=\mathbf{Q}(\alpha_i)$ ($\arccos k:=\mathbf{Q}(\alpha_i)$, resp.) is a number field, 
such that $k\cong \mathbf{Q}(e^{2\pi i\alpha_i+\log\log\varepsilon})$  
($k\cong \mathbf{Q}(\cos 2\pi\alpha_i \times\log\varepsilon)$, resp.)  
Lemmas \ref{lm3.1}, \ref{lm3.3} and Corollary \ref{cor3.4} 
give us a map
\begin{equation}\label{eq3.5}
\mathscr{P}: (\Lambda, [I], \log k)\longrightarrow V(k)
\end{equation}
sending order-isomorphic Handelman triples $(\Lambda, [I], \log k)$
to the $k$-isomorphic projective varieties $V(k)$. 

\medskip
(ii) On the other hand, there exists a map 
\begin{equation}\label{eq3.6}
\mathscr{Q}: V(k) \longrightarrow (\Lambda', [I]', K)
\end{equation}
which sends $k$-isomorphic projective varieties $V(k)$ to the order-isomorphic 
Handelman triples (Theorem \ref{thm2.5}). 

\medskip
(iii) One gets from (\ref{eq3.5}) and (\ref{eq3.6}) a commutative diagram in Figure 1. 
It follows from the diagram that 
$(\Lambda, [I], \log ~k)\cong (\Lambda', [I]',  K)$ are order-isomorphic Handelman triples. 
In particular, the number fields $K\cong \log ~k$ must be isomorphic. 
Therefore, the projective variety $V(k)$  satisfies equation (\ref{eq1.1}).

\begin{figure}
\begin{picture}(300,110)(-80,-5)
\put(45,70){\vector(2,-1){60}}
\put(133,70){\vector(0,-1){35}}
\put(45,83){\vector(1,0){60}}
\put(110,20){$(\Lambda', [I]',  K)$}
\put(-20,80){$(\Lambda, [I], \log ~k)$}
\put(120,80){ $V(k)$}
\put(70,90){$\mathscr{P}$}
\put(142,50){$\mathscr{Q}$}
\put(45,50){$Id$}
\end{picture}
\caption{Handelman's  triples}
\end{figure}

\medskip
Lemma \ref{lm3.5} is proved. 
\end{proof}

\bigskip
Theorem \ref{thm1.1} follows from Lemma \ref{lm3.5}.

\section{Complex multiplication}
We shall  illustrate Theorem \ref{thm1.1} in the simplest case of
an $n$-dimensional abelian variety with  complex multiplication  $A_{CM}^n(k)$
[Lang 1983] \cite{L}.  Roughly speaking, the following result says that $n=r$, 
where $r$ is the rank  of the Drinfeld module.
\begin{corollary}\label{cor4.1}
\begin{equation}
\mathscr{Q}(A^n_{CM}(k))=(\Lambda, [I], ~\log k), 
\end{equation}
where 
\begin{equation}\label{}
\left\{
\begin{array}{cl}
k &\cong  \mathbf{Q}(e^{2\pi i\alpha_1+\log\log\varepsilon}, \dots, e^{2\pi i\alpha_n+\log\log\varepsilon}),  \nonumber\\
\log k & \cong  \mathbf{Q}(\alpha_1,\dots, \alpha_n) .
\end{array}
\right.
\end{equation}
\end{corollary}
\begin{proof}
It is known that $\mathscr{Q}(A^n_{CM})=(\Lambda, [I], K)$,
where $\Lambda=\mathbf{Z}+\alpha_1\mathbf{Z}+\dots+\alpha_n\mathbf{Z}$
\cite[Remark 6.6.2]{N}. It remains to compare this result with the formula
$\Lambda=\mathbf{Z}+\alpha_1\mathbf{Z}+\dots+\alpha_r\mathbf{Z}$
obtained in Lemma \ref{lm3.1} for the Drinfeld modules. Thus one gets $n=r$
in Theorem \ref{thm1.1}. Corollary \ref{cor4.1} follows.
\end{proof}

\begin{example}
Let $n=1$, i.e. $V(k)\cong \mathscr{E}_{CM}$ is an elliptic curve with complex multiplication (Example \ref{exm2.4}). 
In this case $k\cong  \mathbf{Q}(e^{2\pi i\alpha+\log\log\varepsilon})$ and $\log k\cong\mathbf{Q}(\alpha)$, 
where $\alpha$ and $\varepsilon\in\log k$ are irrational quadratic numbers. 
On the other hand,  it is well known that $k\cong \mathbf{Q}(j(\mathscr{E}_{CM}))$ is the Hilbert class field $\mathscr{H}(\mathbf{Q}(\sqrt{-d}))$ of an imaginary quadratic field $\mathbf{Q}(\sqrt{-d})$
corresponding to complex multiplication
with $j(\mathscr{E}_{CM})$ being the $j$-invariant of  $\mathscr{E}_{CM}$. In other words, the algebraic number $e^{2\pi i\alpha+\log\log\varepsilon}$
is a generator of the field  $\mathscr{H}(\mathbf{Q}(\sqrt{-d}))$ distinct from the Weierstrass generator  $j(\mathscr{E}_{CM})$. This fact was first observed in \cite{Nik3}. 
\end{example}

\section*{Data availability}
  
  Data sharing not applicable to this article as no datasets were generated or analyzed during the current study.
   
\section*{Conflict of interest}
On behalf of all co-authors, the corresponding author states that there is no conflict of interest.
  

\section*{Funding declaration}
The author was partly supported by the NSF-CBMS grant 2430454.

\bibliographystyle{amsplain}


\end{document}